\documentclass[11pt]{article}
\usepackage{amsmath, amssymb, theorem, latexsym, epsfig}
 \usepackage{ifpdf}
\numberwithin{equation}{section}

\theoremstyle{plain}
\theorembodyfont{\itshape}
\newtheorem{theorem}{Theorem}[section]
\newtheorem{proposition}[theorem]{Proposition}

\newtheorem{problem}[theorem]{Problem}
\theorembodyfont{\rmfamily}
\newtheorem{definition}[theorem]{Definition}

\newtheorem{example}[theorem]{Example}
\newtheorem{remark}[theorem]{Remark}

\newenvironment{proof}{{\noindent \textbf{Proof}\,\,}}{\hspace*{\fill}$\Box$\medskip}

\def\cc{\mathbb C}
\def\la{\lambda}

\def\oc{\overline{\cc}}

\def\cp{\mathbb{CP}}
\def\wt#1{\widetilde#1}
\def\rr{\mathbb R}

\def\nn{\mathbb N}
 \def\mct{\mathcal T}
 
 \def\mcb{\mathcal B}
 
 \def\rp{\mathbb{RP}}
 \def\mcb{\mathcal B}
 
\def\La{\Lambda}
\def\ii{\mathbb I}
\def\mca{\mathcal A}
  \def\diag{\operatorname{diag}}
 \def\la{\lambda}
\def\mcc{\mathcal C}

\def\mcq{\mathcal Q}

\def\mcn{\mathcal N}

\def\Xi{\mathcal Z}

\title{On complex algebraic caustics in planar and projective billiards}
\author{Alexey Glutsyuk\thanks{Higher School of Modern Mathematics MIPT, Moscow, Russia}
\thanks{HSE University, Moscow, Russia}
\thanks{CNRS, UMR 5669 (UMPA, ENS de Lyon), Lyon, France}
\thanks{Results of the project No 075-00648-25-00 "Symmetry. Information. Chaos.", carried out within the framework of the Basic Research Program at HSE University in 2025 , are presented in this work.}}
 \begin{document}
 \maketitle
 \vspace{-0.3cm}
 \centerline{\it Dedicated to the memory of Alexey Vladimirovich Borisov}
\begin{abstract}   A {\it caustic} of a billiard is a curve whose tangent lines are reflected to its own tangent lines.  A billiard is called {\it Birkhoff caustic-integrable}, if there exists a topological annulus adjacent to its boundary from inside that is foliated by 
closed caustics. The famous Birkhoff Conjecture, studied by many mathematicians, states that {\it the only Birkhoff caustic-integrable billiards are ellipses.} The conjecture is open even for billiards whose boundaries are ovals of algebraic curves. In this case the billiard is known to have a dense family of so-called rational caustics that are also ovals of algebraic curves. 
We introduce the notion of a 
{\it complex caustic:} a complex algebraic curve whose {\it complex} tangent lines are sent by complexified reflection to its own 
complex tangent lines. 
We show that the usual billiard on a real planar curve $\gamma$ has a complex  caustic, if and only if $\gamma$ is a conic. We prove analogous result for billiards on all the surfaces of constant curvature. These results are corollaries of the solution of S.Bolotin's polynomial integrability conjecture: a joint result by M.Bialy, A.Mironov and the author. We extend them 
to the projective billiards introduced by S.Tabachnikov, which are a common generalization of billiards on surfaces of constant curvature. We also deal with a well-known class of projective billiards on conics that are defined to have caustics forming a dual conical pencil. We show that up to restriction to a finite union of arcs, each of them is equivalent to a billiard on appropriate surface of constant curvature. 
  \end{abstract}
  \tableofcontents
  \section{Introduction}
 \subsection{Main result: curves with complex caustics are conics. Plan of the paper}
   Consider a domain $\Omega\subset\rr^2$ bounded by a strictly convex closed smooth curve $C=\partial\Omega$. 
   The billiard reflection acts on the 
oriented lines intersecting $C$. Namely,  an oriented line $L$ is reflected from $C$ at its last intersection point with $C$  according to the standard reflection law: the angle of incidence is equal to the angle of reflection; the reflected line is directed inside  $\Omega$.   Recall that a {\it caustic} of the billiard is a curve 
$\alpha$ such that each tangent line to $\alpha$ is reflected to 
a tangent line to $\alpha$. The billiard  is called 
{\it Birkhoff caustic-integrable,} if  a topological annulus adjacent to $C$ from inside is foliated by closed caustics, and $C$ is a leaf of this foliation.  
 It is well-known that each elliptic billiard is integrable:  ellipses confocal to the boundary are caustics, 
 see \cite[section 4]{tab}. 
 The {\bf Birkhoff Conjecture}  states the converse: {\it the only Birkhoff caustic-integrable 
 convex bounded planar 
 billiards  are  ellipses.} It was studied by many mathematicians including H.Poritsky, 
 M.Bialy,  S.Bolotin, V.Kaloshin, A.Sorrentino, A.Mironov, I.Koval, the author and others.  See its brief survey  in  Subsection 1.7. 
 
Recall that the {\it billiard flow} acts on  $T\rr^2|_\Omega$ as follows. A point $(Q,v)\in T\rr^2|_\Omega$, 
 $Q\in\Omega$, $v\in T_Q\rr^2$, is moved by the geodesic flow:  $Q$ moves with constant velocity $v$, until $Q$ hits  $C$. Then  $v$ is mirror-reflected from 
 the tangent line $T_Q C$ to a new vector $v^*\in T_Q\rr^2$ directed inside $\Omega$, $|v^*|=|v|$. Afterwards  
 $(Q,v^*)$ is moved  by the geodesic flow etc. The billiard flow has  trivial first integral $|v|^2$. Birkhoff integrability of the billiard in $\Omega$ is equivalent to 
 the {\it Liouville integrability of the billiard flow:} to existence of an additional first integral independent with $|v|^2$  on  the intersection with $T\rr^2|_{\overline\Omega}$ of a neighborhood of the unit tangent bundle to the boundary. This is a well-known folklore fact. 
 
  S.V.Bolotin \cite{bolotin, bolotin2} suggested the following {\bf polynomial version of the Birkhoff Conjecture}.  
 {\it Let  the billiard flow in the domain bounded by a $C^2$-smooth closed curve $C$ have a non-trivial first integral polynomial in the velocity:} in this case the billiard is called {\it polynomially integrable.} Then {\it $C$ is an ellipse.} He stated it in more generality, 
 for general, not necessarily closed or convex $C^2$-smooth curves $C$ 
 on  all the surfaces of constant curvature:  plane, round sphere and hyperbolic plane. 
 The solution of Bolotin's Conjecture  is a joint result of M.Bialy, A.Mironov and the author \cite{bm, bm2, gl2}. 
  
 It is well-known that Birkhoff integrability implies existence of a dense family of the so-called {\it rational caustics:} those caustics that are tangent to one-dimensional families of periodic orbits. 

 The Birkhoff Conjecture is open for billiards with  boundary being an oval of algebraic curve. It is known that in this case the rational caustics are ovals of algebraic curves. Moreover, to the author's knowledge, the following problem is also open.
 
\begin{problem} \label{pb1} Does there exist a pair of two nested ovals  of planar algebraic curves, distinct from confocal ellipses, such that the smaller one is a caustic for the bigger one?
\end{problem}

In the present paper we prove  positive answer under a stronger assumption saying that the complexification of the 
caustic is a complex caustic of the real billiard, see the next definition and theorem.

Let $C\subset\rr^2\subset\cc^2\subset\cp^2$ be  a 
$C^1$-smooth planar curve. For every point $Q\in C$ consider the reflection involution 
$\rp^1\to\rp^1$ acting on  lines through $Q$: the reflection from the tangent line $T_QC$. It extends to a complex projective involution $\cp^1\to\cp^1$ acting on  complex lines through $Q$. 
\begin{definition} Let $C$ be as above. 
Let $\alpha\subset\cp^2$ be a complex algebraic curve. 
The curve $\alpha$ is said to be a {\it complex caustic} of the billiard on $\gamma$, if for 
every $Q\in C$ every complex tangent line to $\alpha$ through $Q$ is reflected as above 
to a complex tangent line to $\alpha$. 
\end{definition}

\begin{theorem} \label{talg1} Let $C\subset\rr^2$ be a nonlinear $C^2$-smooth connected embedded curve, not necessarily closed, convex or algebraic.  Let $C$ 
have a  complex caustic $\alpha$. Then $C$ is a conic, and $\alpha$ is either a confocal conic to $C$, or a finite 
union of confocal conics. 
\end{theorem}
 \begin{remark} The curve $C$ is a real caustic for itself: through each its point $Q$ passes the tangent line to $C$, and it is fixed by the reflection at $Q$. If $C$ is a conic, then its complexification $C_\cc$ is a complex  
caustic for $C$ for the same reason. But  if $C$ is a higher degree algebraic curve, then 
a priori its complexification $C_\cc$ is not necessarily a complex caustic. Namely, for a generic point $P\in C_\cc$ besides the line $L_P$ tangent to $C_\cc$ at $P$ there is at least one more other line through $P$ tangent to $C_\cc$ at some other point.  
To check, whether $C_\cc$ is a complex caustic for the curve $C$, one has to check 
whether for every $Q\in C$ the collection of all the complex lines through $Q$ tangent to $C_\cc$ is invariant under the reflection at $Q$. This is a non-trivial condition on the algebraic curve $C$. 
\end{remark}
\begin{remark} A priori if a real algebraic curve $\alpha$ is a caustic for a closed curve $C$, it is not true that its complexification is a complex caustic for $C$. Indeed, let $\alpha$ be a real oval of an algebraic curve. The string construction, see \cite[p.73]{tab},   yields a family of nested strictly convex closed curves $C_t$, $t\geq0$, $C_0=\alpha$, such that $\alpha$ is a real caustic for each $C_t$. If $\alpha$ is not an ellipse, then  the  curve $C_t$ is not an ellipse for generic $t$. Then the complexification of the curve $\alpha$ is not its complex caustic, by Theorem \ref{talg1}. The question of Problem \ref{pb1} asks whether the same can happen for $\alpha$ and $C_t$ being both  ovals of algebraic curves.
\end{remark}

In what follows for an algebraic curve $\alpha\subset\cp^2$ by $\wt\alpha$ we denote the compact Riemann surface  parametrizing $\alpha$ holomorphically, bijectively except for self-intersections. Every point $A\in\wt\alpha$ is identified with the corresponding point of the algebraic curve $\alpha$. For every $A,B\in\wt\alpha$ by $AB$ we denote the line through the 
corresponding points of the algebraic curve $\alpha$. 
\begin{definition} Let $C,\alpha\subset\cp^2$ be two projective algebraic curves.   Recall that the {\it tangential correspondence} is the algebraic subset $\mct\subset\wt C\times\wt\alpha$ consisting of those points $(A,B)$ for which the line $AB$ is tangent to 
$\alpha$ at $B$. In the special case, when $C=\alpha$, the tangential correspondence always contains the diagonal 
$\{(A,A) \ | \ A\in C\}$ as an irreducible component. In this case the union of its non-diagonal irreducible components will be denoted by $\mct$ and 
called the {\it non-diagonal tangential correspondence.}
\end{definition}
In what follows for a real curve $\alpha$ its complexification, i.e., its complex Zariski closure in $\cp^2$, will be denoted by the 
same symbol $\alpha$. 
\begin{proposition} \label{prtang} Let $C,\alpha\subset\rr^2\subset\cc^2\subset\cp^2$ be nested strictly convex ovals of algebraic curves, $C$ be the bigger one. Let $\alpha$ be a caustic for the billiard inside the curve $C$. Let $\wt C$, $\wt\alpha$ be normalizations of 
their complexifications. Let the tangential correspondence in $\wt C\times\wt\alpha$ be irreducible. In the case, when 
 $C$ and $\alpha$ have the same complex Zariski closure, we require irreducibility of the non-diagonal tangential correspondence.  Then $\alpha$ is a complex 
caustic for the billiard on $C$, and hence, $C$ and $\alpha$ are confocal ellipses.
\end{proposition}
Proposition \ref{prtang} is proved in Subsection 2.4.

\begin{problem} Is it true that if the billiard inside a strictly convex oval $C$ of a real algebraic curve is Birkhoff integrable, then 
the complexification of at least one its rational caustic is a complex caustic?
\end{problem}
Positive answer would imply the statement of the Birkhoff Conjecture for ovals of algebraic curves.

In the next subsection we state Theorem \ref{talg2}, which generalizes  Theorem \ref{talg1} for billiards on surfaces of constant curvature. 

In 1997 S.Tabachnikov \cite{tabpr} introduced projective billiards, which generalize billiards on surfaces of constant curvature. 
He suggested a generalized version of the Birkhoff Conjecture for projective billiards and for their projective-dual objects: the dual billiards \cite{tab08}. See the next two definitions and the corresponding background material in Subsections 1.3 and 1.4.

\begin{definition} \cite{tabpr} A {\it planar projective billiard} is a $C^1$-smooth curve $C\subset\rr^2$ equipped with a transversal line field 
$\mcn$. For every $Q\in C$ the {\it projective billiard reflection involution} at $Q$ acts on the space 
 $\rp^1$ of lines through $Q$. Namely, it permutes two lines $a$, $b$ through $Q$, if  and only if 
 the quadruple of lines $T_QC$, $\mcn(Q)$, $a$, $b$ is harmonic. The above reflection  is the unique projective involution that fixes $T_QC$, $\mcn(Q)$ and permutes $a$, $b$.
 \end{definition}
  \begin{remark} \label{reminv} The above involution is the projectivization of a linear involution $\mathbf I_Q:T_Q\rr^2\to T_Q\rr^2$ having eigenlines $T_QC$, $\mcn(Q)$ with eigenvalues 1 and -1 respectively. It defines the projective reflection law at $Q$: an incoming vector 
  $v_{in}\in T_Q\rr^2$ is reflected to outcoming vector $v_{out}=\mathbf I_Q(v_{in})$. 
 For  a projective billiard  on a closed and strictly convex curve $C$  the {\it projective billiard map} \cite{tabpr} acts on the {\it phase cylinder:} 
 the space of {\it oriented} lines intersecting $C$. It sends an oriented line to its image under the above reflection  at its last intersection point  with $C$ in the sense of orientation. See Fig. 1.
 \end{remark}
 \begin{figure}[ht]
  \begin{center}
  \vspace{-1.4cm}
   \epsfig{file=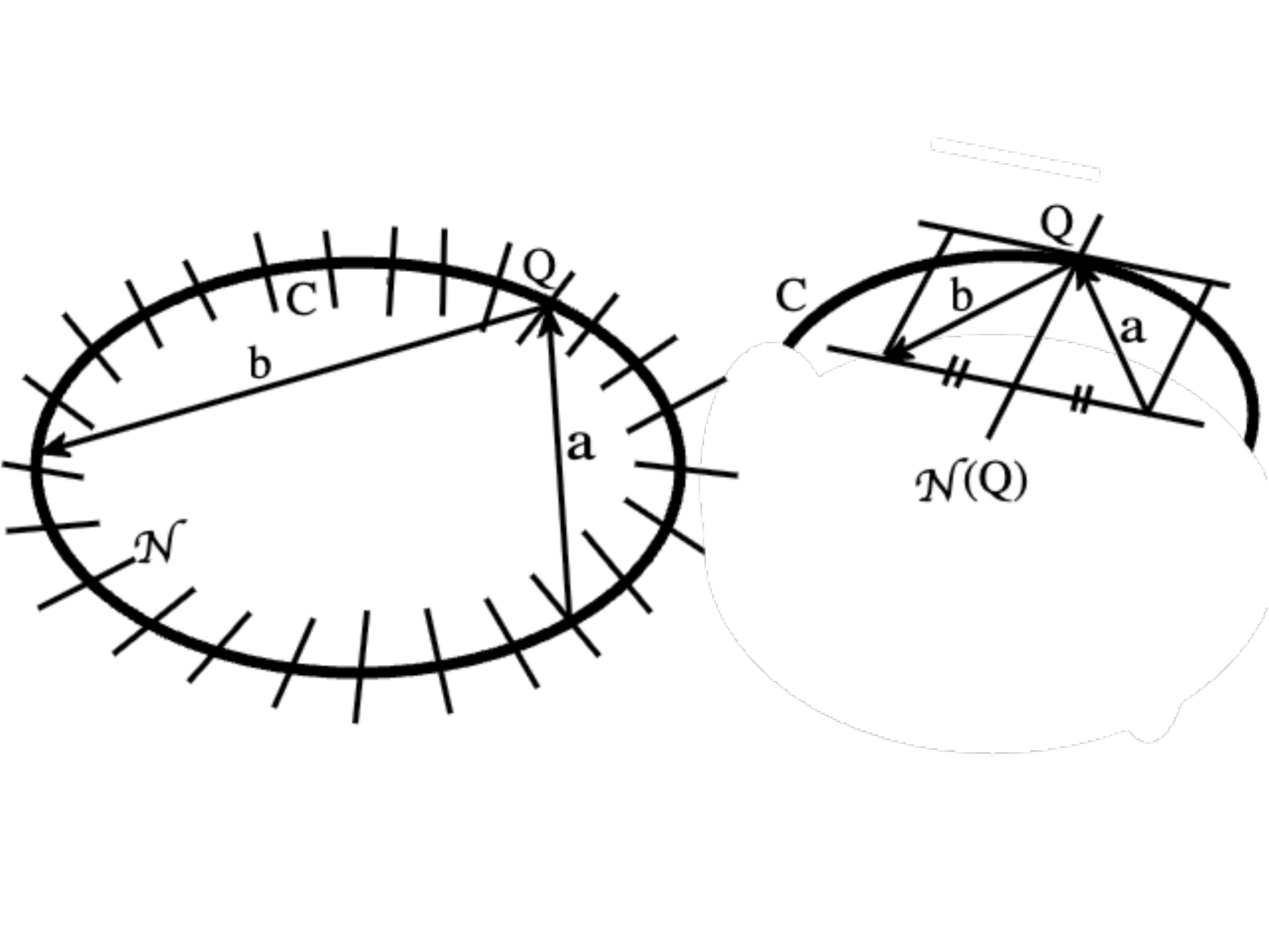, width=22em}
   \vspace{-1.5cm}
    \caption{The projective billiard reflection.}
    \vspace{-0.5cm}
    \label{fig:0}
  \end{center}
\end{figure}
A usual planar billiard with reflections from a curve $C$  is a particular case of projective billiard, with $\mcn$ being the normal line field. For representation of billiards on other surfaces of constant curvature as projective billiards see Subsection 1.3

In Subsection 1.5 we state Theorem \ref{talg3}, which extends Theorems \ref{talg1} and \ref{talg2} to projective billiards 
with two complex caustics. It states that then the underlying curve $C$ is a conic and the projective billiard flow is rationally 0-homogeneously integrable: has a first integral that is a rational 0-homogeneous function of the velocity of uniformly bounded degree. Rationally 0-homogeneously integrable projective billiards on non-linear connected $C^4$-smooth curves exist only on conics, 
 and they are classified up to projective transformations in \cite{grat}. This implies that the projective billiard in question belongs 
 to the list given by the latter classification. 
 
The standard class of rationally integrable projective billiards are the so-called called {\it dual pencil type} projective billiards, 
see Example \ref{defdp}. 
Each of them is the  projective billiard structure on a conic $C$ defined to have a prescribed conical caustic $\alpha\neq C$. 
In Subsection 1.6 we state  Proposition \ref{ppencil}, which says that up to projective transformation and restriction to a finite union of arcs, each dual pencil type projective billiard  is equivalent to a billiard on a surface of constant curvature.
 
 To outline the proofs, let us recall the two following definitions.
\begin{definition} \label{defdual} Let $\gamma\subset\rp^2$ be a planar curve. For every point $P\in\gamma$ by $L_P$ we denote the projective tangent line to $\gamma$ at $P$.  A  planar {\it dual billiard} is a planar curve $\gamma$ equipped with a family 
$(\sigma_P)|_{P\in\gamma}$ 
of projective involutions of the tangent lines $L_P$ that fix $P$.  
\end{definition}
\begin{definition} A planar dual billiard is {\it rationally integrable,} if there exists a non-constant rational function $R$ of two variables  whose restriction to each tangent line $L_P$ is $\sigma_P$-invariant: $(R\circ\sigma_P)|_{L_P}=R|_{L_P}$.
\end{definition}

 {\bf Sketch of proof of Theorem \ref{talg1}.}
Projective duality given by the orthogonal polarity, see the beginning of Subsection 1.4,  transforms the curve $C$ to its dual curve $\gamma=C^*$. 
The usual billiard with reflections from the curve $C$ is transformed to a special dual billiard on the curve $\gamma$: 
the Bialy--Mironov angular billiard in $\rp^2_{[M_1:M_2:M_3]}$, see Subsection 1.4. 
The dual to a complex caustic $\alpha$ is  an {\it invariant curve} $\alpha^*\subset\cp^2$ of the dual billiard, which means that 
$\sigma_P$ permutes the points of intersection $\alpha^*\cap L_P$ for every $P\in\gamma$.  For the proof of Theorem \ref{talg1} we show that 
appropriate rational function having zero locus $\alpha^*$ and polar locus the absolute $\{ M_1^2+M_2^2=0\}$ is an integral 
of the dual billiard. It takes the form 
$R(M_1,M_2)=\frac{H(M_1,M_2)}{(M_1^2+M_2^2)^d}$, where $H$ is a degree $2d$ homogeneous polynomial vanishing 
exactly on $\alpha^*$. 
Rational integrability of the dual billiard is equivalent to rational 0-homogeneous integrability of the initial billiard:  
see \cite[proposition 1.27]{grat} recalled below as Proposition \ref{procriter}. For usual planar billiards the latter is equivalent to polynomial integrability, see
 \cite[proposition 1.8]{gl3}. This implies that the  billiard inside $C$ is polynomially integrable. A direct deduction of polynomial integrability, without using results of \cite{grat, gl3}, is presented in 
the discussion on p. 1004 in \cite{gl2}.  
Polynomial integrability    implies that $C$ is a conic, by the solution 
of the Bolotin's  Conjecture \cite{bm, bm2, gl2}. Afterwards we prove that $\alpha$ is a finite union of conics confocal to $C$. 
This will prove Theorem \ref{talg1}. 

Proofs of  generalizations, i.e.,   Theorems \ref{talg2}, \ref{talg3},   are  analogous: the same projective duality transforms projective billiards to dual billiards and preserves rational (0-homogeneous) integrability, by \cite[proposition 1.27]{grat}.


\subsection{Extension to billiards on surfaces of  constant curvature}  
We rescale the metric on the surface in question by constant factor so that the curvature is equal to either $0$, or $\pm1$. It is well-known that then the surface can be realized as a surface $\Sigma\subset\rr^3_{x_1,x_2,x_3}$ equipped with appropriate quadratic form $<\mca dx,dx>$; here $\mca$ is appropriate real symmetric 3$\times$3-matrix and $<\ , \ >$ is the Euclidean scalar product. The metric of the surface coincides with the restriction of the quadratic form to its tangent planes. 

1) The Euclidean plane: $\Sigma=\rr^2=\{ x_3=1\}$, $<\mca dx,dx>=dx_1^2+dx_2^2$.

2) The sphere: $\Sigma=\{ x_1^2+x_2^2+x_3^2=1\}$, $\rr^3$ is Euclidean, i.e., $\mca=Id$.

3) The hyperbolic plane: $\Sigma=\{ x_3^2-x_1^2-x_2^2=1\}\cap\{ x_3>0\}$, $\rr^3$ is Minkowski, i.e., 
$<\mca dx, dx>=dx_1^2+dx_2^2-dx_3^2$.

Recall that the geodesics of the metric on $\Sigma$ are exactly its intersections with two-dimensional vector subspaces in 
$\rr^3$. Consider the restriction to $\Sigma$ of the tautological projection $\pi:\rr^3\setminus\{0\}\to\rp^2_{[x_1:x_2:x_3]}\subset\cp^2$. The geodesics are sent to projective lines. 

Let $C\subset\Sigma$ be a $C^1$-smooth curve, $Q\in C$. The billiard reflection involution from the curve $C$ at $Q$ acts on the 
space $\rp^1$ of geodesics through $Q$ in $\Sigma$. The latter space is identified with the projectivization of the tangent plane 
$T_Q\Sigma$. It  is identified by the projection $\pi$ with the space of projective lines through $\pi(Q)$, which is also identified with $\rp^1$. The projection $\pi$ conjugates the above billiard 
reflection involution to a projective involution acting on the lines through $\pi(Q)$. Its complexification is a projective involution 
$\cp^2\to\cp^2$ of the space of complex lines through $\pi(Q)$. It is called the {\it complex reflection} from $\pi(C)$ at $\pi(Q)$. 
 
\begin{definition} Let $\Sigma$ and $C$ be  as above.  Let $\alpha\subset\cp^2$ be a complex algebraic curve. 
The curve $\alpha$ is a {\it complex caustic} for the curve $C$, if for every point $Q\in C$ every complex line $L$ through 
$\pi(Q)$ tangent to $\alpha$ is send by the above complex reflection to a line tangent to $\alpha$.
\end{definition}
 Recall that the {\it absolute} $\mathbb I\subset\cp^2$  of a surface $\Sigma$ as above is the light curve, the zero locus of the corresponding quadrartic form: 

Case 1) of plane: $\mathbb I=\{ x_1^2+x_2^2=0\}$ is a union of two lines through $(0,0)$.

Case 2) of sphere: $\ii=\{ x_1^2+x_2^2+x_3^2=0\}$.

Case 3) of hyperbolic plane: $\ii=\{ x_1^2+x_2^2-x_3^3=0\}$. 

\begin{remark} \label{iicaust} In Cases 2), 3) the  absolute $\ii$ is a complex caustic for every billiard on $\Sigma$. Indeed,  take any point $Q\in\Sigma$ and any germ of $C^1$-smooth curve through $Q$. Let  $L$ denote the geodesic  through $Q$ tangent to $C$. 
 The reflection of geodesics through $Q$ with respect to $L$  is the restriction to $\Sigma$ of a linear involution $J:\rr^3\to\rr^3$ fixing points of $L$ and preserving  the quadratic form defining $\Sigma$. Among the 2-dimensional complex vector subspaces in $\cc^3$ through $Q$ there are exactly two that are tangent to the light cone $\pi^{-1}(\ii)$. They are exactly those subspaces on which 
 the quadratic form defining $\Sigma$ is degenerate. Therefore, their union is $J$-invariant.  Thus, the  absolute  
 is a complex caustic. In fact, the latter subspaces are permuted by reflection. This follows from \cite[corollary 2.3 and proposition 2.5]{gl2}, by duality given by the orthogonal polarity presented at the beginning of Subsection 1.4. 
\end{remark}

Recall that a curve $C\subset\Sigma$ is called a {\it conic,} if it is the intersection of the surface $\Sigma$ with a quadratic cone: zero locus of a homogeneous quadratic polynomial. 
The complexified projectivization of the latter cone is a complex conic in $\cp^2$ containing the projection $\pi(C)$. In what follows we identify $C$ and the latter complex conic. 

Two conics in $\cp^2_{[x_1:x_2:x_3]}$ are called {\it confocal} with respect to the quadratic form $<\mca dx,dx>$, if there exists a symmetric matrix $\mcb$ such that the  conics in question belong to the dual pencil consisting of the conics 
$$\mcc_\la=\{<(\mcb-\la\mca)^{-1}x,x>=0\}, \ \ \la\in\oc.$$
Confocality of conics in $\Sigma$ or of a conic $C\subset\Sigma$ and a complex conic $\alpha\subset\cp^2$ is defined 
as confocality of the corresponding complex conics. See \cite[p. 84]{veselov2} and the above identification.
\begin{theorem} \label{talg2} Let $\Sigma$ be one of the above surfaces. Let $C\subset\Sigma$ be a nonlinear $C^2$-smooth connected embedded curve, not necessarily closed, convex or algebraic.  Let $C$ 
have a  complex caustic $\alpha$, and let in Cases 2), 3) $\alpha$ be distinct from the absolute $\ii$. Then $C$ is a conic, and $\alpha$ is either a confocal conic to $C$, or a finite union of conics confocal to $C$.   
\end{theorem}
 
\subsection{Projective billiards and rational integrability}
Projective billiards generalize billiards on surfaces of constant curvature, see \cite{tabpr}. Namely, a planar curve 
equipped with the projective billiard structure given by normal line field is a usual billiard.  
Consider now a connected curve $C$ in a surface $\Sigma$ of constant curvature that is disjoint from the plane $\{ x_3=0\}$. 
Then it is bijectively projected onto the curve $\pi(C)$ lying in the affine chart $\rr^2=\{ x_3\neq0\}$. The normal line field 
on $C$ in the metric  of the surface $\Sigma$ is projected by $d\pi$ to a transversal line field on $\pi(C)$:  
a projective billiard structure on $\pi(C)$. Each billiard orbit in $\Sigma$ 
with reflection from the curve $C$ is projected by $\pi$ to a projective billiard orbit.

Recall that for every point $Q$ of a curve $C\subset\rp^2$ by $L_Q$ we denote the projective tangent line to $C$ at $Q$. 

The {\it flow of a projective billiard} is defined in \cite{tabpr}  analogously to the Euclidean billiard, but with reflection of velocities given by the projective  reflection. Namely,  let $\Omega\subset\rr^2_{x_1,x_2}$ be a domain with smooth boundary 
 $C=\partial\Omega$ equipped with a projective billiard structure.  Given a point 
 $(Q,v)\in T\rr^2$, $Q\in\Omega$, $v=(v_1,v_2)\in T_Q\rr^2$, the flow moves the point $Q$ 
with uniform velocity $v$  until it  hits the boundary 
$C$ at some point $H$. 
Take the image of the velocity vector $v$ under its translation to $H$. 
Take its reflection image $v^*=\mathbf I_H(v)\in T_H\rr^2$, see Remark \ref{reminv}. Afterwards 
the flow moves the point $H$ with the new uniform velocity $v^*$ until its orbit 
hits $C$  etc. See Fig. 2. 

Birkhoff integrability of projective billiard is equivalent to existence of a non-trivial  first integral of the billiard flow that is 
$0$-homogeneous in velocity and whose restriction to the intersection of 
$T\rr^2|_{\overline\Omega}$ with a neighborhood of the unit tangent bundle to $\partial\Omega$ does not have critical points 
\cite{grat}. This is analogous to the above similar folklore fact on the usual billiards. 
 \begin{figure}[ht]
  \begin{center}
  \vspace{-0.7cm}
   \epsfig{file=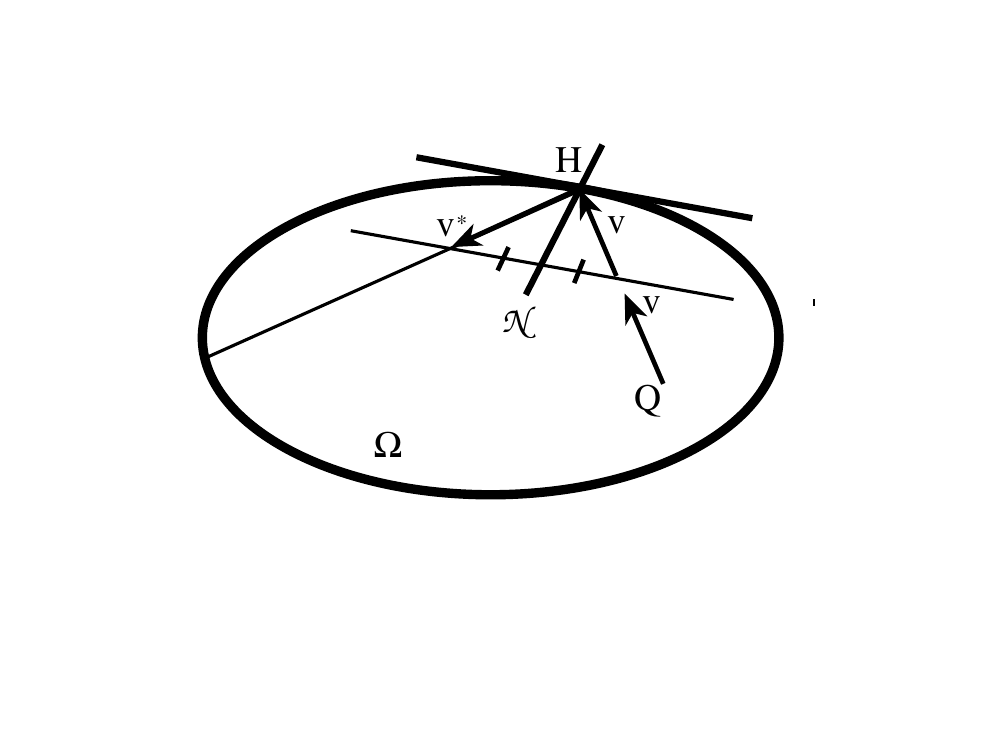, width=20em}
   \vspace{-2cm}
   \caption{Projective billiard flow}
   \vspace{-0.7cm}
        \label{fig:0}
  \end{center}
\end{figure}
 \begin{definition} \cite{grat}. A planar projective billiard is {\it rationally 0-homogeneously 
 integrable,} if its flow admits a non-constant 
 first integral $I$ of the type 
 $$I(Q,v)=\frac{I_{1,Q}(v)}{I_{2,Q}(v)}; \ \ \ I_{1,Q}(v), \ I_{2,Q}(v) 
 \text{ are homogeneous polynomials,}$$
$$\deg I_{1,Q}=\deg I_{2,Q}, \ \ \ \ \text{ the degrees } \deg I_{j,Q} \text{ are uniformly bounded in } Q.$$
It is called a {\it rational 0-homogeneous integral:}  $I(Q,\la v)\equiv I(Q,v)$ for $\la\in\rr$.
 \end{definition}
 The notion of a caustic or a complex caustic of a projective billiard is the same, as for the usual billiard, but now with 
reflection being the projective billiard reflection. 
\begin{remark} \label{rlev} Each rational 0-homogeneous integral of a projective billiard, if any, is always a rational 0-homogeneous function of the moment vector 
$$M=[r,v]=(-v_2,v_1,\Delta), \ \ r=(x_1,x_2,1), \ \ v=(v_1,v_2,0), \ \ \Delta=x_1v_2-x_2v_1,$$
see \cite{grat}. This is a generalization of S.Bolotin's observation \cite{bolotin2}. 
In the above integrable case  each  level set $\{ I=c\}$ is the union of tangent bundles of a one-parameter family of lines. This follows by 0-homogeneity  of the integral $I$ and its invariance under the billiard flow.  The enveloping curve of the latter line family is a caustic of the projective billiard. 
\end{remark}
\begin{example} \label{defdp} \cite{grat} Let $C\subset\rr^2$ be a non-linear and non-single-point real conic. 
Let $S\subset\cp^2$ be a complex conic defined by a real quadratic equation. Let 
$\mathcal B\subset C$ denote the subset of points  $X\in C$ for which the complex tangent line to $C$ at $X$ 
is also tangent to $S$. It consists of at most four points and is called the {\it base set} of the dual pencil generated by $C$ and $S$.  
For every point $Q\in C^o:=C\setminus\mathcal B$ consider the linear involution $\mathbf I_Q: T_Q\rr^2\to  T_Q\rr^2$ having eigenline $T_QC$ with eigenvalue 1 and whose complexification permutes the complex  lines through $Q$  tangent to  $S$.  It has a real eigenline $\mcn(Q)$  with eigenvalie -1. The projective billiard on $C^o$ defined by the eigenline field $\mcn$ is called a {\it dual pencil type  billiard.} It has a family of conical caustics whose  projective dual 
 conics lie in a pencil  containing the conics dual to $C$ and $\alpha$.  
The above line $\mcn(Q)$ is exactly the line that passes through $Q$ and through the pole of the tangent line to $C$ at $Q$ with respect to the conic $S$. See  \cite[section III]{ccs} and 
 \cite[section 2.3]{fierobe-th} for the two latter statements.
 
\end{example}
 It was shown in  \cite{grat} that rationally 0-homogeneously integrable projective billiard exist only on conics. They were 
 classified there up to projective equivalence, i.e., up to projective transformation, see the next theorem.
 
\begin{theorem} \label{tgermproj} \cite{grat} {\bf A.} Let $C\subset\rr^2_{x_1,x_2}$ be a non-linear $C^4$-smooth germ of 
 curve at a point $O$, equipped with a transversal line field $\mcn$. 
  Let the corresponding germ of projective billiard be $0$-homogeneously rationally integrable. 
 Then $C$ is a conic. The line field $\mcn|_{C\setminus\{ O\}}$ extends to a global analytic line field   on the  conic $C$. It is    transversal to $C$ except for at most four tangency points:  projective billiard singularities. The extended field $\mcn$  has one of the following types up to real projective equivalence.
 
 1) The field $\mcn$ defines a dual pencil type  billiard, see Example \ref{defdp}.

2)  An {\bf exotic billiard:} $C=\{x_2=x_1^2\}\subset\rr^2_{x_1,x_2}\subset\rp^2$, and 
the  line field $\mcn$  is directed by one of the following vector fields at points of  $C$: 
 \smallskip
 
2a) \ \ \ \ \ $(\dot x_1,\dot x_2)=(\rho,2(\rho-2)x_1)$, 
$$\rho=2-\frac2{2N+1} \text{ (case 2a1)), \ or }  \ \ \rho=2-\frac1{N+1} \text{ (case 2a2)), }  \ \ N\in\mathbb N,$$
 the  vector field 2a) has quadratic first integral $\mcq_{\rho}(x_1,x_2):=\rho x_2-(\rho-2)x_1^2.$

2b1) \ $(\dot x_1, \dot x_2)=(5x_1+3,
2(x_2-x_1))$, \ \ \ 2b2)  \ $(\dot x_1,\dot x_2)=(3x_1, 2x_2-4)$, 
\smallskip

2c1) \ $(\dot x_1, \dot x_2)=(x_2, 
x_1x_2-1)$, \ \ \ \ \ \ 2c2) \ $(\dot x_1, \dot x_2)=(2x_1+1, 
x_2-x_1)$.

2d) \ $(\dot x_1, \dot x_2)=(7x_1+4, 2x_2-4x_1)$.

{\bf B.} The tangency points of  the field $\mcn$  with the complex conic $C$ are: 

Case 1):  the base points of the dual pencil.

Case 2a): $(0,0)$ and the infinite point $E$ of the parabola $C$. 

Cases 2b1), 2c2), 2d): $(0,0)$, $(-1,1)$, $E$.

Case 2b2): $(\pm i,-1)$, $E$. 

Case 2c1): $(-e^{\frac{2\pi i j}3}, e^{\frac{4\pi i j}3})$,  $j=0,1,2$. 
\end{theorem}
{\bf Addendum to Theorem \ref{tgermproj} \cite{grat}.} {\it The projective billiards from Theorem \ref{tgermproj} have the 
following $0$-homogeneous rational integrals, called {\bf canonical integrals:}

Case 1): A ratio of two homogeneous quadratic polynomials in the moment vector $M=(-v_2, v_1, \Delta)$, $\Delta:=x_1v_2-x_2v_1$. 
The above polynomials vanish on two conics dual to some two conics of the dual pencil defining $\mcn$. Here the projective duality is given by the orthogonal polarity, see the beginning of Subsection 1.4.

Case 2a1), $\rho=2-\frac2{2N+1}$:
\begin{equation}\Psi_{2a1}(M):=\frac{(4v_1\Delta-v_2^2)^{2N+1}}{v_1^2\prod_{j=1}^N(4v_1\Delta-c_jv_2^2)^2}, \ \ \ 
c_j=-\frac{4j(2N+1-j)}{(2N+1-2j)^2}. \label{r2a1v}\end{equation} 

Case 2a2), $\rho=2-\frac1{N+1}$:
\begin{equation}\Psi_{2a2}(M)=\frac{(4v_1\Delta-v_2^2)^{N+1}}{v_1v_2\prod_{j=1}^N(4v_1\Delta-c_jv_2^2)}, \ \ \ 
c_j=-\frac{j(2N+2-j)}{(N+1-j)^2}. \label{r2a2v}\end{equation}

Case 2b1): 
 \begin{equation}\Psi_{2b1}(M)=\frac{(4v_1\Delta-v_2^2)^2}{(4v_1\Delta+3v_2^2)(2v_1+v_2)(2\Delta+v_2)}.\label{r2b1v}\end{equation}
 
 Case 2b2): 
 \begin{equation}\Psi_{2b2}(M)=\frac{(4v_1\Delta-v_2^2)^2}{(v_2^2+4\Delta^2+
4v_1\Delta+4v_1^2)(v_2^2+4v_1^2)}.
 \label{r3b2v}\end{equation}
 
 \begin{equation}\text{Case 2c1): } \ \ \ \ \ \ \  \ \ \ \ \Psi_{2c1}(M)=\frac{(4v_1\Delta-v_2^2)^3}{(v_1^3+\Delta^3+
 v_1v_2\Delta)^2}.\label{r2c1v}\end{equation}
 
 Case 2c2): 
  \begin{equation}\Psi_{2c2}(M)=\frac{(4v_1\Delta-v_2^2)^3}{(v_2^3+2v_2^2v_1+(v_1^2+2v_2^2+5v_1v_2)\Delta+v_1\Delta^2)^2}.
 \label{r2c2v}\end{equation}
 
 Case 2d):}  $\Psi_{2d}(M)$
 \begin{equation}=\frac{(4v_1\Delta-v_2^2)^3}{(v_1\Delta+2v_2^2)(2v_1+v_2)(8v_1v_2^2+2v_2^3+(4v_1^2+5v_2^2+28v_1v_2)\Delta+16v_1\Delta^2)}.
 \label{r2dv}\end{equation}

\subsection{Projective and dual billiards. Duality and integrability}
The {\it orthogonal polarity} in the Euclidean space $\rr^3_{x_1,x_2,x_3}$ sends a two-dimensional vector subspace $W$ 
to its orthogonal complement $W^\perp$: a one-dimensional vector subspace. Its projectivization, which sends 
the projective line $\pi(W\setminus\{0\})$ to the point $\pi(W^\perp\setminus\{0\})$, is a projective duality 
$\rp^{2*}\to\rp^2$ sending lines to points, which is also called orthogonal polarity. Every curve $C\subset\rp^2$ is transformed 
to its dual curve $C^*$, which consists of the points dual to the tangent lines to $C$. The above duality transforms 
the space of lines through a given point $Q\in C$ to the line tangent to the dual curve $\gamma:=C^*$ at the point $P=P(Q)=L_Q^*\in \gamma$. 

For every line $L\subset\{ x_3=1\}=\rr^2_{x_1,x_2}$ denote by $W_L\subset\rr^3$ the two-dimensional 
vector subspace containing $L$.  Take an arbitrary vector $v=(v_1,v_2)$ directing $L$. Then $W^{\perp}_L$ is generated by the 
moment vector $M=[r,v]$, see Remark \ref{rlev}. Thus, the point in $\rp^2$ dual to $L$ has homogeneous coordinates 
$[M_1:M_2:M_3]$. Therefore, the above projective duality sends lines in $\rp^2_{[x_1:x_2:x_3]}$ to points in 
$\rp^2_{[M_1:M_2:M_3]}$.

Let now a curve $C$ lie in the affine chart $\rr^2=\{ x_3\neq0\}$ identified with the plane $\{ x_3=1\}$. Let $C$  be equipped with a projective billiard structure. Let 
$Q\in C$. It is well-known that then the above duality conjugates the projective billiard reflection involution at $Q$ with a projective involution of the line $Q^*=L_P$ tangent to $\gamma=C^*$ at $P$, see \cite{tabpr, grat}.  
This yields a dual billiard structure on $\gamma$. Conversely, each dual billiard structure on a planar curve that does not pass through the origin  is dual to a projective billiard. 

A curve $S$ is a caustic for a projective billiard on a curve $C$, if and only if 
its dual $S^*$ is an {\it invariant curve} for the corresponding dual billiard: by definition, this means that for every $P\in \gamma$ the involution $\sigma_P$ permutes the intersection points of the line $L_P$ with $S^*$. See Fig. 3. 
  \begin{figure}[ht]
  \begin{center}
  \vspace{-0.8cm}
   \epsfig{file=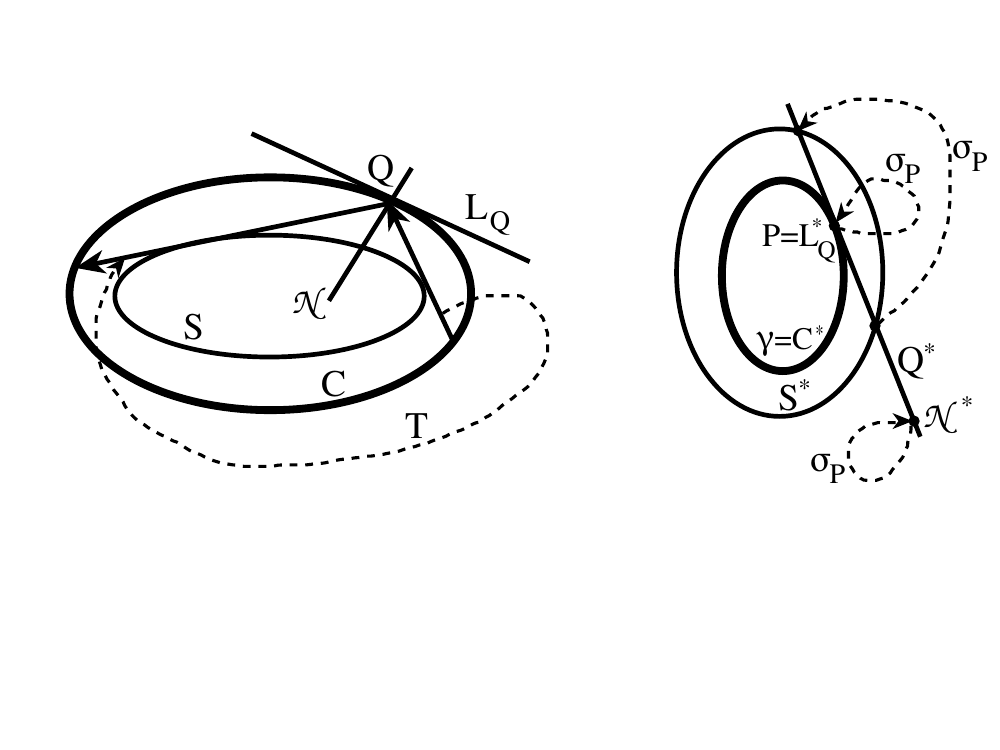, width=23em}
   \vspace{-2.5cm}
    \caption{The projective billiard reflection involution $T$ acting on lines through a point $Q\in C$ and the dual involution $\sigma=\sigma_{P}$ acting on  the dual line $Q^*$ tangent to the dual curve $\gamma=C^*$ at the point $P=L_Q^*$.}
    \label{fig:0}
    \vspace{-0.5cm}
  \end{center}
\end{figure}
 \begin{example} Consider a {\it central projective billiard} on a planar curve $C$, where all the lines of the transversal field pass through the origin. Its dual is the {\it outer  billiard} on the dual curve $\gamma$, see 
  \cite{tabpr}. Namely, $\gamma$ lies in the standard affine chart $\rr^2=\{ M_3\neq0\}\subset\rp^2_{[M_1:M_2:M_3]}$ and the dual billiard involutions $\sigma_P:L_P\to L_P$ are affine central symmetries with respect to $P$.  
  \end{example}
  \begin{example} \label{exbm} The dual to the usual Euclidean billiard on a curve $C\subset\rr^2$ is the {\it Bialy--Mironov angular billiard} on  $\gamma=C^*\subset\rp^2$. The corresponding involutions $\sigma_P:L_P\to L_P$, $P\in\gamma$, called {\it Bialy--Mironov angular symmetries,} are defined as follows. Let $O=(0,0)\in\rr^2$. Note that for every $P\in\gamma$ one has $OP\neq L_P$:  in the contrary case 
 the point $Q=L_P^*\in C$ would be infinite, which is impossible. By definition, 
  $\sigma_P$ fixes $P$, and the involution $\sigma_P$ permutes  points 
 $a^*, b^*\in L_P$ if and only if  the lines $Oa^*$, $Ob^*$  are symmetric with respect to the line $OP$.  
 See \cite{bm} and Fig. 4. 
 \end{example}

 \begin{figure}[ht]
  \begin{center}
  \vspace{-1.2cm}
   \epsfig{file=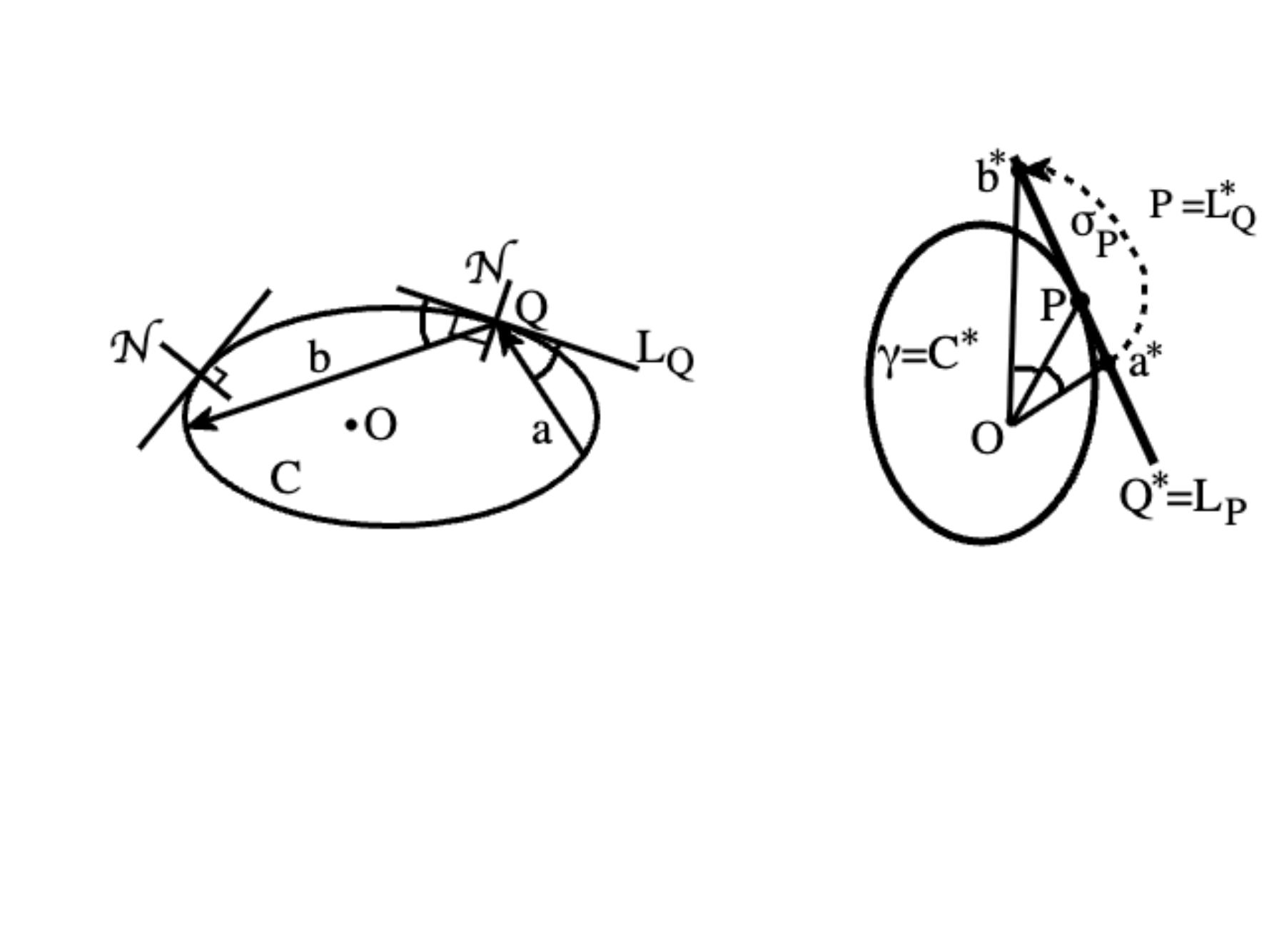, width=25em}
   \vspace{-3.1cm}
   \caption{Euclidean  billiard and its dual: Bialy -- Mironov angular billiard. Here it is obtained by the polar duality with respect to the unit circle, which is the  orthogonal polarity post-composed with the central symmetry.}
        \label{fig:4}
        \vspace{-0.3cm}
  \end{center}
\end{figure}
\begin{example} \label{exdp} Let conics $C$, $S$ and punctured conic $C^o$ be the same, as in Example \ref{defdp}. 
Consider the dual pencil type projective billiard structure on $C^o$ defined there to have $S$ as a caustic. Let $\gamma=C^*, 
S^*\subset\cp^2_{[M_1:M_2:M_3]}$ be their dual conics with respect to the orthogonal polarity. Let 
$$\gamma=\{<UM,M>=0\}, \ \ S^*=\{<AM,M>=0\}.$$
Then the corresponding dual billiard is defined on the punctured conic $\gamma^o$: the conic $\gamma$ punctured at the points dual to the common tangent lines to $C$ and $S$. For every $P\in\gamma^o$ the  involution $\sigma_P:L_P\to L_P$ permutes the points of intersection of the line $L_P$ with the conic $S^*$. Moreover, the latter holds with 
$S^*$ replaced by every conic from the pencil of conics 
$$\mcc_\la:=\{<(U-\la A)M,M>=0\}, \ \ \la\in\oc; \ \ \mcc_0=\gamma,$$
by Desargues' Theorem,  see  \cite{berger87}. The duality given by the orthonogal polarity transforms the pencil $\mcc_\la$ 
to the dual pencil of conics containing $C$ and $S$:
\begin{equation}\mcc_\la^*:=\{\mca_\la x,x>=0\}, \ \ \mca_\la:=(U-\la A)^{-1}.\label{fordp}\end{equation}
Its complex conics are caustics for the projective billiard on $C$, by duality. 
\end{example}

\begin{proposition} \label{procriter}\footnote{A generalization of Proposition \ref{procriter} for piecewise smooth projective billiards is given in \cite[proposition 3.3]{gl3}.} \cite{grat}
 1)  A planar projective billiard on a  $C^2$-smooth curve is rationally 0-homogeneously 
 integrable with integral of degree $n$, if and only if its dual billiard  is rationally integrable with integral of degree $n$.
 
 2)  If $R(M)$ is a rational $0$-homogeneous function 
 such that $R[r,v]$ is a degree $n$ first integral of the projective billiard flow, then 
 $R(M)$ is a degree $n$ rational integral of the dual billiard. 
 
 3)  Conversely, if $R(M)$ is a rational integral of 
 the dual billiard,  then $R([r,v])$ is a $0$-homogeneous rational integral 
 of the projective billiard. 
 \end{proposition}
 \begin{remark} Versions of Proposition \ref{procriter}  for polynomially integrable billiards on surfaces of constant curvature 
 were earlier proved respectively in the paper \cite{bolotin2} by S.V.Bolotin   and in two joint papers by  M.Bialy and A.E.Mironov 
 \cite{bm, bm2}; see also \cite[theorem 2.8]{gl2}. 
 \end{remark}
\subsection{Generalization of Theorem \ref{talg1} to projective billiards}

If a projective billiard has a rational 0-homogeneous integral $I=I(M)$, $M=[r,v]$, then each its level hypersurface 
 is a union  of complex tangent bundles of complex one-dimensional family of complex lines. 
The enveloping curve of the latter family of lines is a  complex  caustic,  as in Remark \ref{rlev}. It is the orthogonal-polar-dual curve to a level curve of the integral $I(M)$ considered as a function of $[M_1:M_2:M_3]\in\cp^2$, which is a rational integral of the corresponding dual billiard. See Proposition \ref{procriter}.

\begin{theorem} \label{talg3} Let $C$ be a nonlinear $C^4$-smooth connected  embedded planar curve. Let  
$C$ be equipped with a projective billiard structure having at least two different complex  caustics. 
Then $C$ is a conic, and the projective billiard is rationally 0-homogeneously integrable. Thus, 
up to projective transformation, the projective billiard on $C$ belongs to the list from Theorem \ref{tgermproj} and has the corresponding canonical integral $\Psi$ from its addendum. Each complex caustic  is  a finite union of enveloping curves corresponding to level hypersurfaces of the integral $\Psi$, see the above paragraph. 
 In other words, it is a finite union of curves that are dual to some level curves 
of the integral $\Psi(M)$ as a function on $\cp^2_{[M_1:M_2:M_3]}$ under the orthogonal polarity. 
\end{theorem}

\subsection{Projective billiards of dual pencil type as billiards on surfaces of constant curvature}

Recall that we identify the  plane $\rr^2_{x_1,x_2}$ with the plane $\{ x_3=1\}\subset\rr^3_{x_1,x_2,x_3}$: 
$$\text{to each point } \ x=(x_1,x_2)\in\rr^2 \ \text{ corresponds the point } \ r=(x_1,x_2,1).$$ 
The above plane is identified with the affine chart $\{ x_3\neq0\}\subset\rp^2_{[x_1:x_2:x_3]}$. 
Each affine line $L\subset\rr^2=\{ x_3=1\}$ is contained in a unique two-dimensional vector subspace in $\rr^3$. This is the subspace whose projectivization is $L$. 

Everywhere below for a curve $C\subset\rr^2$ and a point $x\in C$ by 
$H_C(x)\subset\rr^3$ we denote the above subspace corresponding to $L=T_xC$. For every line field $\mcn$ on $C$ 
by $H_{\mcn}(x)$ we denote the subspace corresponding to $L=\mcn(x)$. 

\begin{proposition} \label{ppencil} 1) Every dual pencil type projective billiard  on a conic $C\subset\rr^2_{x_1,x_2}=\{ x_3=1\}\subset\rr^3$ punctured in at most four base points, see Example \ref{defdp},  
is given by a transversal line field $\mcn$ defined by a real symmetric $3\times3$-matrix $\mca$ as follows. 
For every $x\in C$ the line $\mcn(x)$ is defined by the condition that the two-dimensional subspaces $H_C(x)$ and  $H_{\mcn}(x)$ in $\rr^3$ are {\bf $\mca$-orthogonal,} i.e., orthogonal in the scalar product 
$<\mca x,x>$.

2) Statement 1) holds for every real matrix $\mca$ 
 such that the complex conic $\alpha_\mca:=\{<\mca x,x>=0\}\subset\cp^2$ lies in the 
dual pencil defining the projective billiard and is different from $C$, see Example \ref{exdp}. 
In particular, the matrix $\mca$ can be chosen non-degenerate.

3) Up to projective transformation and restriction to a finite union of open arcs, every dual pencil  projective billiard is the projection of a  billiard on a surface of constant curvature, see the beginning of Subsection 1.3. 
\end{proposition} 
\begin{remark} \label{remort} In the case, when $\alpha_\mca=C$, for every $x\in C$ the corresponding vector $r$ is $\mca$-orthogonal to the subspace $H_C(x)$. For every line $\mcn(x)$ through $x$ the corresponding subspace $H_{\mcn}(x)$ is 
$\mca$-orthogonal to $H_C(x)$.
\end{remark}

\subsection{Historical remarks}
V.Lazutkin \cite{laz} proved that every strictly convex bounded planar billiard with sufficiently smooth boundary has uncountably many, continuum of  closed caustics.  H.Poritsky \cite{poritsky}, and later 
E.Amiran \cite{amiran} 
 proved the Birkhoff Conjecture under the additional assumption that for every two nested caustics the smaller one is a caustic for the billiard in the bigger one. M.Bialy \cite{bialy} proved that if the phase cylinder of a billiard is 
 foliated  by non-contractible invariant closed curves,  then the billiard boundary is a circle. See also \cite{wojt} for another proof.  
  Bialy proved  similar results  for billiards on  constant curvature surfaces \cite{bialy1} and  for magnetic billiards 
 \cite{bialy2}.  D.V.Treschev's numerical experience indicated that 
 there should exist billiards with reflections from two germs of analytic curves 
 having a  two-periodic point where the  germ of the second iterate of reflection is analytically conjugated to a disk rotation:  see \cite{treshchev} and also  \cite{tres2,tres3}. 
   V.Kaloshin and A.Sorrentino \cite{kalsor} proved that {\it an integrable deformation of 
 an ellipse is an ellipse,} under the condition that the family of caustics extends up to a caustic tangent to triangular orbits. See also \cite{kavila} for a similar previous result for ellipses with small excentricities. A generalization of the result of Kaloshin and Sorrentino, without extension condition but for a generic ellipse, was recently proved by I.Koval \cite{koval}. M.Bialy and A.E.Mironov \cite{bm6} 
 proved the Birkhoff Conjecture for centrally-symmetric 
 billiards having a family of closed caustics that extends up to a caustic tangent to four-periodic orbits. 
 For a survey on the Birkhoff Conjecture and related results see \cite{bols, kalsor, KS18, bm6} and 
 references therein. 
 
It was shown by the author \cite{glcaust} that every convex $C^{\infty}$-smooth non-closed planar 
 curve with positive curvature has an adjacent domain that admits a  continuum of 
 distinct $C^\infty$-smooth foliations by non-closed caustics, with the boundary being a leaf. This extends R.Melrose's result  \cite{melrose1}.

 A.P.Veselov obtained a series of complete integrability results  and description of billiard orbits via a shift of Jacobi variety for billiards bounded by confocal quadrics 
 in space forms of any dimension. 

For a survey on the Bolotin's polynomial version of the Birkhoff  Conjecture, its version for magnetic billiards and related results 
see  \cite{bm, bm2, bm5}  and references therein and also the book \cite{kozlov}.

M.Berger have shown that in Euclidean space $\rr^n$ with $n\geq3$ the only hypersurfaces admitting caustics are quadrics \cite{berger87}. This result was extended to space forms of constant curvature of dimension greater than two by the author of the 
present paper \cite{commute}. 

S.Tabachnikov \cite{tabpr}  introduced projective billiards and proved a criterium and a necessary condition 
for a planar projective billiard to preserve an area form. He proved that every  projective billiard on the unit circle preserves one and the same, but infinite canonical area form. He proved that if a projective billiard on a circle
preserves an area form that is smooth up to the boundary of the phase cylinder, then the billiard is integrable. 

It was shown in \cite{gs} that each polynomially integrable outer billiard is a conic. Generalizations of this result to  rationally integrable dual and projective billiards and their classifications were obtained in \cite{grat, gl3}.
%

\section{Proofs of main results}
In the proofs of Theorems \ref{talg1}, \ref{talg2} we use the next proposition.
\begin{proposition} \label{ratpol}  \cite[proposition 1.8]{gl3}. A non-polygonal billiard with piecewise $C^2$-smooth boundary 
on a surface of constant
curvature is rationally 0-homogeneously integrable, if and only if it is polynomially integrable. Then the minimal degree of a rational 0-homogeneous integral is equal to either 2, or 4: to the minimal degree of a non-trivial nonlinear polynomial integral.
\end{proposition}
\begin{remark}  \label{rratpol} In the above theorem in the case, when the boundary is connected and $C^2$-smooth, it is a conic, by the solution of the 
Bolotin's Conjecture \cite{bm, bm2, gl2}. Then the minimal degree of a non-linear polynomial integral is equal to two, see  \cite{bolotin2}.  Therefore, the minimal degree of a rational 0-homogeneous integral is also equal to two. 
\end{remark}

\subsection{Case of planar billiards. Proof of Theorem \ref{talg1}}
Let $C$ be a $C^2$-smooth curve in the Euclidean plane $\rr^2_{x_1,x_2}$ identified with the affine chart $\{ x_3=1\}$ in 
$\rp^2_{[x_1:x_2:x_3]}$, 
and let $\gamma=C^*\subset\rp^2_{[M_1:M_2:M_3]}$ be its dual curve with respect to the orthogonal polarity. 
Recall that the dual to the billiard on the curve $C$ is the dual billiard on $\gamma$ given by Bialy--Mironov 
angular symmetries $\sigma_P:L_P\to L_P$, $P\in\gamma^*$, see Example \ref{exbm}. 

\begin{proposition} \label{pinvlp} Let $C$, $\gamma$, $\sigma_P$ be as above. 
 Let  the billiard on $C$ have a complex caustic  $\alpha$. Let $\alpha^*$ be its   projective dual, and let 
$H(M_1,M_2,M_3)$ be its defining homogeneous polynomial: $\alpha^*=\{ H=0\}$, and $H$ has the minimal possible 
degree. Set 
\begin{equation} d:=\deg H, \ \  R(M_1,M_2,M_3):=\frac{H^2(M_1,M_2,M_3)}{(M_1^2+M_2^2)^{d}}.\label{zdiv}
\end{equation}
The function $R$, which is a well-defined rational function on $\rp^2_{[M_1:M_2:M_3]}$,  
is an integral of the dual billiard on $\gamma$.
\end{proposition}
\begin{proof} Set 
$$\ii:=\{ M_1^2+M_2^2=0\}\subset\cp^2_{[M_1:M_2:M_3]}.$$
For every $P\in\gamma$ the complexification of the angular symmetry $\sigma_P:L_P\to L_P$ is the projective involution of 
 the complexified  line $L_P$ that fixes the point $P$ and permutes its intersection points with $\ii$, see \cite{bm}, \cite[proposition 2.18]{gl2}. Thus, it leaves invariant  the pole locus $L_P\cap \ii$  of the restriction $R|_{L_P}$ and its zero locus 
 $L_P\cap \alpha^*$. This implies that the pole and zero divisors of the restriction $R|_{L_P}$ are $\sigma_P$-invariant. 
 Therefore, the restriction $R|_{L_P}$ and its pullback under the involution $\sigma_P$ have the same pole and zero divisors. 
 Thus, the involution $\sigma_P$ either fixes it, or changes its sign, since two rational functions of one variable having the same 
 pole and zero  divisors differ by a constant factor. For a generic $P\in\gamma$ one has $R(P)\neq0,\infty$, and $\sigma_P$ is well-defined and fixes $P$. Hence, $(R\circ\sigma_P)(P)=R(P)$, and thus, $R|_{L_P}$ is $\sigma_P$-invariant. This implies that 
 $R$ is an integral of the dual billiard and proves Proposition \ref{pinvlp}.
\end{proof} 

Thus, the dual billiard structure on $\gamma$ is rationally integrable. Therefore, the initial billiard  is rationally 0-homogeneously integrable, by Proposition \ref{procriter}. Hence, it is polynomially integrable, by Proposition \ref{pinvlp}. Thus, $C$, and hence 
$\gamma$ is a conic.

Let us now prove that $\alpha$ is a finite union of conics confocal to $C$.  The minimal degree of a rational 0-homogeneous integral of the billiard on $C$ is equal to two, see Remark \ref{rratpol}. 
Hence, the minimal degree of a rational integral of the  Bialy--Mironov angular billiard on $\gamma$ 
is also equal to two, by Proposition \ref{procriter}. Let $\mcc$ denote the pencil of conics containing $\gamma$ and the absolute $\ii$.  
Each conic in $\mcc$ is an invariant curve for the dual billiard on $\gamma$: 
for every $P\in\gamma$ the involution $\sigma_P:L_P\to L_P$, whenever well-defined, permutes the intersection points. 
This follows from the same statement for the conics $\gamma$ and $\ii$ generating the pencil $\mcc$ 
and Desargues' Theorem, see  \cite{berger87}. This also follows from the well-known facts that confocal conics to $C$ are caustics of the 
billiard on $C$ and the projective duality given by the orthogonal polarity transforms them to 
conics of the pencil $\mcc$, see, e.g., \cite{bm, gl2}.  Therefore, every quadratic rational function $\Psi$ whose level curves are 
the conics of the pencil $\mcc$ is an integral of the angular billiard on $\gamma$. But the angular billiard has yet another 
rational integral $R$, whose zero locus is the algebraic curve $\alpha^*$. On the other hand, $R$ is constant on each 
irreducible component of every level curve of the quadratic integral  $\Psi$, see \cite[proposition 2.7]{gl3}. 
Thus, $\alpha^*$ is a finite union of conics of the 
pencil $\mcc$, and $\alpha$ is a finite union of confocal conics to $C$. Theorem \ref{talg1} is proved.

\subsection{Billiards on surfaces of constant non-zero curvature. Proof of Theorem \ref{talg2}} 
Let $\Sigma\subset\rr^3_{x_1,x_2,x_3}$ be a surface from Subsection 1.2 of non-zero constant curvature: either the unit sphere in the Euclidean space $\rr^3$, or the upper half of pseudosphere of squared radius -1 in the Minkowski space. 
Let  $\ii\subset\cp^2$ be the corresponding absolute. Recall that it is 
 a regular conic, and it is a complex caustic of every billiard on $\Sigma$, 
see Remark \ref{iicaust}. 

Let $C\subset\Sigma$ be a  $C^2$-smooth curve. For simplicity let us first consider that $C$ 
is bijectively tautologically projected onto a curve in the standard  affine chart $\{ x_3\neq0\}\subset\rp^2$, also denoted 
by $C$. This is always the case if $\Sigma$ is the hyperbolic plane. 
Consider the billiard on $C$ as a projective billiard, see the beginning of Subsection 1.3. 
 The projective duality given by the orthogonal polarity preserves the absolute $\ii$, i.e., it is self-dual, and 
 transforms the billiard on  $C\subset\Sigma$ to a special dual billiard 
 on the dual curve $\gamma=C^*$: the so-called $\ii$-angular billiard, see \cite[subsection 2.1]{gl2}. Namely, for 
 every $P\in \gamma$ the corresponding complexified involution $\sigma_P:L_P\to L_P$ is the unique projective 
 involution fixing $P$ and permuting the points of intersection $L_P\cap\ii$. Note that the duality between billiards in $\Sigma$ 
 and $\ii$-angular billiards, as presented in \cite[subsection 2.1]{gl2}, does not use the 
 assumption that the projected curve lies in the standard affine chart. 

Let now the billiard on a $C^2$-smooth connected curve $C\subset\Sigma$ have a complex caustic $\alpha\neq\ii$. 
Then the corresponding $\ii$-angular billiard has two different complex invariant curves: $\ii$ and $\alpha^*$. 
Without loss of generality we can and will consider that each irreducible component of the curve $\alpha^*$ 
is different from $\ii$. Indeed, replacing $\alpha^*$ by the union of its components different from $\ii$ yields a complex caustic, 
since $\ii$ is a complex caustic and the number of lines tangent to $\ii$ and to a component of $\alpha^*$ different from 
$\ii$ is finite. Fix a rational 0-homogeneous 
function $R(M_1,M_2,M_3)$ on the ambient projective plane of the dual billiard that has poles and zeros exactly on $\ii$ 
and $\alpha^*$ respectively, so that the order of zero is the same on all the irreducible components of the curve $\alpha^*$. 
For every point $P\in\gamma$ the restriction $R|_{L_P}$ is $\sigma_P$-invariant, as in the previous subsection. Therefore, 
$R$ is an integral of the $\ii$-angular billiard. Therefore, the billiard on $C$ is rationally 0-homogeneously integrable, by 
Proposition \ref{procriter}. Hence, it is polynomially integrable, by Proposition \ref{ratpol}. 
Therefore, $C$ is a conic, and the  billiard has a quadratic integral, 
by the solution of the Bolotin's  Conjecture 
\cite{bm, bm2, gl2}. The dual billiard on $\gamma$ has a quadratic rational integral whose level curves are 
conics forming a pencil containing $\gamma$ and $\ii$, as in the previous 
subsection. Their dual conics are conics confocal to $C$. Then $\alpha^*$ is a union of a finite number of conics of the latter pencil, which follows from \cite[proposition 2.7]{gl3}, as in the previous subsection. 
Hence, $\alpha$ is a finite union of complex conics confocal to $C$. Theorem \ref{talg2} is proved.

 \subsection{Case of projective billiards. Proof of Theorem \ref{talg3}}
 Let a projective billiard on a $C^4$-smooth curve $C\subset\rr^2_{x_1,x_2}=\{ x_3=1\}\subset\rr^3_{x_1,x_2,x_3}$ 
 have two different complex  caustics $\alpha_1$ and $\alpha_2$. Then the corresponding dual billiard on $\gamma=C^*$  has two different complex algebraic invariant curves $\alpha_1^*,\alpha_2^*\subset\cp^2_{[M_1:M_2:M_3]}$. 
 Without loss of generality we consider that the  curves $\alpha_1^*$, $\alpha_2^*$ have no common irreducible component. Indeed,  the union $A$ of their common irreducible components is a complex caustic. Therefore, replacing each $\alpha_j^*$  
 by the union of its components that are not contained in $A$ yields a complex caustic, which will now be denoted 
 by $\alpha^*_j$, as in the previous subsection. 
  Fix a rational 0-homogeneous function $R(M_1,M_2,M_3)$, whose pole and zero divisors as of a function on $\cp^2$ are 
formal sums of irreducible components of the curves  $\alpha_1^*$ and $\alpha_2^*$ respectively, multiplied by  coefficients 
$c_1\in-\nn$, respectively $c_2\in\nn$. 
Then it is a rational integral of the dual billiard, as in the previous subsection. 
Therefore, the projective billiard on $C$ is rationally 0-homogeneously integrable, by Proposition \ref{procriter}. 
Hence, $C$ is a conic, and up to projective transformation, the projective billiard belongs to the list given by  
 Theorem \ref{tgermproj}. The corresponding function $\Psi(M)$ from its Addendum, treated as a function 
 on $\cp^2_{[M_1:M_2:M_3]}$, is a rational integral of the dual billiard, by Proposition \ref{procriter}. It is called a {\it canonical integral} of the dual billiard, see \cite[section 2]{gl3}. Every other rational integral of the dual billiard, in particular $R$, is constant on each irreducible component of every level curve of the canonical 
 integral $\Psi$, by \cite[proposition 2.7]{gl3}.  Its generic level curve  is irreducible, by 
 \cite[lemma 2.8]{gl3}. Therefore $R$ is constant on every its level curve, since each reducible level curve is a limit of irreducible ones. Hence, each $\alpha_j^*$ is 
 a finite union of level curves of the canonical integral $\Psi(M)$. Thus, its dual curve $\alpha_j$ is a finite union of  curves dual to its level curves. Theorem \ref{talg3} is proved. 

\subsection{Tangential correspondence. Proof of Proposition \ref{prtang}}
The tangential correspondence $\mct$ is complex one-dimensional.  Let $\mct_0\subset\mct$ denote the subset consisting of those pairs $(A,B)$ 
for which the image of the line $AB$ by reflection from the curve $C$ at $A$ is a line tangent to $\alpha$. Note that in order that the reflection at $A$ be well-defined, we have to exclude from $\wt C$ the isotropic tangency points: those points $A$ at which the tangent line to $C$ is either the infinity line, or an isotropic line, i.e.,  parallel to some of the vectors $(1,\pm i)$, see \cite{alg}. 
Let $\wt C'$ denote the curve $\wt C$ punctured at the latter points. Let $\mct'$ denote the correspondence $\mct$ 
punctured in the finite set of those points $(A,B)$ for which $A$ is an isotropic tangency point.  
The subset $\mct_0$ is an analytic subset in  $\wt C'\times\wt\alpha$. It contains the real one-dimensional family of points $(A,B)$ in the products of the real ovals such that the line $AB$ is tangent to $\alpha$ at $B$. Indeed, real tangent lines to $\alpha$ are reflected to its real tangent lines, since $\alpha$ is a real caustic. Therefore,  $\mct_0$ is a one-dimensional analytic subset contained in an irreducible one-dimensional analytic set $\mct'$. Hence, $\mct_0=\mct'$, and thus, for every real point $A\in C$ and 
every $B\in\wt\alpha$ such that the $AB$ is tangent to $\alpha$ at $B$ the  reflection at $A$ sends $AB$ to 
a line tangent to $\alpha$. Thus, $\alpha$ is a complex caustic. Proposition \ref{prtang} is proved.

\subsection{Dual pencil type projective billiards. Proof of Proposition \ref{ppencil}}

Let $C\subset\rr^2=\{ x_3=1\}\subset\rp^3$ be a regular conic. Let $S\subset\rr^2$ be another regular conic, $\mathcal B\subset C$ 
be the subset of those points $X$ such that the line $L_X$ tangent to $C$ at $X$ is also  tangent to $S$. 
Consider the corresponding dual pencil type projective billiard on $C^o=C\setminus\mathcal B$, defined to have $S$ as a caustic. 
Let $\mca$ denote the real symmetric 3$\times$3-matrix defining a conic $\alpha_\mca$ from the dual pencil generated by $C$ and $S$: one has $\alpha_\mca=\{<\mca x,x>=0\}$. Without loss of generality we consider that $\alpha_\mca\neq C$, see 
Remark \ref{remort}.

Consider the space $\rr^3$ equipped with the quadratic form $<\mca dx,dx>$. Let $\mcn$ denote 
the transversal line field on $C^o$ defining the above dual pencil type projective billiard. Recall that for every $X\in C^o$ we denote by  
$H_\mcn(X), H_C(X)\subset\rr^3$  the two-dimensional vector subspaces containing respectively the line $\mcn(X)$ 
and the line $L_X$. Their intersection is the one-dimensional vector subspace containing $X$. It will be  denoted by $\La_X$. 

\smallskip
{\bf Claim 1.} {\it For every $X\in C^o\setminus S$ the subspaces $H_{\mcn}(X)$, $H_C(X)$ are $\mca$-orthogonal, i.e., orthogonal with respect to the quadratic form  $<\mca x,x>$.}

\begin{proof} Without loss of generality we can and will consider that $\La_X$ is not self-$\mca$-orthogonal, i.e., 
$<\mca x,x>\not\equiv 0$ on $\La_X$. This is true for a generic point $X\in C^o$, since $\alpha_\mca\neq C$. 
 The general statement then follows by 
passing to limit. 
Let us consider the ambient plane $\rr^2$ of the projective billiard as the standard affine chart $\{ x_3\neq0\}\subset\rp^2_{[x_1:x_2:x_3]}$. The projective billiard reflection at $X$ acts on  lines through $X$ by an affine involution $\rr^2\to\rr^2$  fixing  the points of the line $L_X$. It  is the projectivization of a linear involution $I_X:\rr^3\to\rr^3$ fixing the points 
of the subspace $H_C(X)$ and fixing the subspace $H_\mcn(X)$. The involution $I_X$ has  double eigenvalue 1 corresponding to the invariant subspace $H_C(X)$ and a simple eigenvalue -1 corresponding to a one-dimensional vector subspace in 
$H_\mcn(X)$, since it acts on the quotient $H_\mcn(X)\slash\La_X\simeq\mcn(X)$ as the central symmetry $v\mapsto-v$. Its complexification permutes the two complex 2-dimensional vector subspaces $\Pi_{\pm}(X)\subset\cc^3$ containing $\La_X$ that are tangent to the complex cone  $K:=\{<\mca x,x>=0\}\subset\cc^3$. Indeed,  $\Pi_{\pm}(X)$ are distinct, since 
$\La_X\not\subset K$. The involution $I_X$ fixes their union,  since $\alpha_A$ is a caustic. If, to the contrary, $I_X$ fixed each of them, then the above 
reflection involution would fix three distinct complex lines through $X$: the complexified line $L_X$ and the projectivizations of the subspaces $\Pi_{\pm}(X)$. Hence, it acts trivially on the space $\cp^1$ of lines through $X$ in $\cc^2\subset\cp^2$ and 
fixes the points of the line $L_X$.  Thus, it is the identity, -- a contradiction, since it acts on $\mcn(X)$ by central symmetry.

{\bf Claim 2.} {\it Every linear involution $I:\rr^3\to\rr^3$ fixing points of the line $\La_X$ and permuting $\Pi_{\pm}(X)$ is an isometry of the quadratic form 
$<\mca x,x>$.}

\begin{proof} The subspaces $\Pi_{\pm}(X)$ are exactly those two-dimensional subspaces $\Pi$ containing $\La_X$ on which 
the form $<\mca x,x>$ is degenerate. The kernel of its restriction to $\Pi$ is  the one-dimensional subspace along which $\Pi$ is 
tangent to $\alpha_A$. It is different from the line $\La_X$, by assumption. Thus, the union  $\cup_{\pm}\Pi_{\pm}(X)$ is uniquely defined by  $X$ and the quadratic form. 
The involution $I$ is non-trivial, and  $\La_X$ is its eigenline with unit eigenvalue. The eigenvalues of its action on the 
quotient $\rr^3\slash \La_X$ are 1 and -1. Indeed, otherwise they would be equal, and hence, the involution $I$ would fix each $\Pi_{\pm}(X)$, 
which is impossible. Let $V\subset\rr^3$ denote its two-dimensional eigensubspace corresponding to the eigenvalue 1. It contains $\La_X$ and it is distinct from $\Pi_{\pm}(X)$, by its invariance.  Hence, the form $<\mca x,x>$ is non-degenerate on $V$.

There exists a unique non-trivial linear involution $J:\rr^3\to\rr^3$ that preserves the quadratic form and fixes the points of the subspace $V$: it acts on its $\mca$-orthogonal one-dimensional subspace as the central symmetry $v\mapsto-v$. It also preserves the union of the subspaces $\Pi_{\pm}(X)$, by the above argument. 
The involution $J$  permutes the subspaces $\Pi_{\pm}(X)$, as in the arguments before Claim 2. 
Finally, the involutions $I$ and $J$ both fix the points of the two-dimensional subspace $V$ and permute $\Pi_{\pm}(X)$. 
Thus, their composition $I\circ J$ fixes $\Pi_{\pm}$, and it fixes each point in $V$. Therefore, $I\circ J=Id$, as in the  argument before Claim 2. 
Hence, $I=J$, and $I$ preserves the quadratic form. The claim is proved.
\end{proof}

Thus, the reflection from the curve $C$ at $X$ is given by an involution $J:\rr^3\to\rr^3$ preserving the quadratic form that has   eigenvalues 1, 1, -1. The double eigenvalue 1 corresponds to the subspace $H_C(X)$, and 
-1 corresponds to its $\mca$-orthogonal complement.  But the subspace $H_{\mcn}(X)$ is also invariant and different from $H_C(X)$. Therefore,  the eigenvalues there are 1 and -1. Hence, $H_{\mcn}(X)$ contains the above $\mca$-orthogonal 
complement to $H_C(X)$, and hence, is $\mca$-orthogonal to $H_C(X)$. Claim 1 is proved.
\end{proof}

Claim 1 implies Statements 1) and 2) of Proposition \ref{ppencil}. 

Let us prove Statement 3). Consider the dual pencil defining the projective billiard on $C$. 
Recall that every real symmetric 3$\times$3-matrix $\mca$ defining a conic of the dual pencil 
 satisfies Statement 1). 

Case 1): the above matrix $\mca$ can be chosen so that the quadratic form $<\mca x,x>$ is sign-definite. 
We consider that it is positive definite, multiplying it by $\pm1$. 
 Then we can achieve that $\mca=Id$, applying a real linear change of coordinates to the quadratic form. 
 In the new coordinates, denoted by $\wt x_1,\wt x_2,\wt x_3$,  for every $X\in C$ the line  $\mcn(X)$ is defined by the condition that the subspace $H_\mcn(X)$ is Euclidean-orthogonal to  $H_C(X)$. Consider the tautological projection of the unit sphere 
 $S^2$ in the Euclidean space $\rr^3_{\wt x_1, \wt x_2, \wt x_3}$ centeted at the origin to the projective plane $\rp^2$. 
Fix an arbitrary connected component $\wt C\subset S^2$ of  the projection preimage of the curve $C$. 
Then the normal line field on $\wt C$ is projected to 
 the line field $\mcn$ on $C$. Thus, the projective billiard on $C$ is the projection of the spherical billiard with reflections 
 from the curve $\wt C$. 
  
 Case 2): the quadratic form $<\mca x,x>$ cannot be chosen sign-definite. Let us choose a regular conic $\alpha\neq C$ from the dual pencil  so that the conic $C$ either intersects the convex domain  $W\subset\rp^2$ bounded by $\alpha$, or entirely lies there. Set $C':=C\cap W$. 
 Let $\mca$ be a real symmetric 3$\times$3-matrix defining $\alpha$, i.e., $\alpha=\alpha_\mca=\{<\mca x,x>=0\}$. 
 We can and will consider that $\det\mca<0$, multiplying $\mca$ 
 by $\pm1$.  Then we can and will consider that $\mca=\diag(1,1,-1)$, 
 applying a real linear change of coordinates. In the new coordinates, 
 let us denote them by $\wt x_1,\wt x_2,\wt x_3$, the conic $\alpha$ becomes the projectivized absolute of the surface 
 $\Sigma=\{ \wt x_1^2+\wt x_2^2-\wt x_3^2=-1\}\subset\rr^3$ of constant negative curvature. The curve $C'$ is the tautological 
 projection of a curve $\wt C'\subset\Sigma$. The projective billiard on $C'$ defined by the dual pencil coincides with the projection of the Riemannian billiard on $\wt C'$ in the metric of the surface $\Sigma$. Namely,  for every $\wt X\in\wt C'$, set 
 $X=\pi(\wt X)\in C'$, the line $\mcn(X)$ is the image under the differential $d\pi$ of the normal line to $\wt C'$ through $X$ in $T_{\wt X}\Sigma$. 
 This follows from Statement 1) rewritten in the new coordinates: the line $\mcn(X)$ is defined by the condition that in the above notations the two-dimensional subspaces $H_{\mcn}(X)$ and $H_C(X)$ are orthogonal with respect to the quadratic form 
 $\wt x_1^2+\wt x_2^2-\wt x_3^2$. Statement 3) is proved. The proof of Proposition \ref{ppencil} is complete. 
 
 \begin{remark} In the above arguments we realized the projective billiard on $C$, or on a finite union of its open arcs, as a billiard on a curve in a surface of {\it non-zero} constant curvature. Let us make a remark on how sometimes one can get the case of zero curvature, i.e., the case of planar billiard. Statement 1) of Proposition \ref{ppencil} holds for all matrices $\mca_\la=(U-\la A)^{-1}$ from (\ref{fordp}), which 
 define conics of the dual pencil. For certain values $\la=\la_0$ the corresponding matrix $U-\la_0 A$ may become degenerate, so that its inverse $\mca_{\la_0}$ is not well-defined. Suppose that $\operatorname{rank}(U-\la_0A)=1$, thus, two 
 eigenvalues of the matrix $U-\la A$ converge to zero, as $\la\to\la_0$. Suppose that their ratio converges to a non-zero limit. Then as $\la\to\la_0$, dividing the matrix $\mca_\la$ by 
 one of the latter eigenvalues, we get a matrix family converging to a nonzero limit matrix $\mca$. 
 The latter matrix $\mca$  also satisfies Statement 1) of Proposition \ref{ppencil}. It is  
 real symmetric and degenerate of rank two. Thus, it has two non-zero eigenvalues. In the special case, when they have the same sign, we can and will consider that $\mca=\diag(1,1,0)$, applying a linear change of coordinates and multiplying $\mca$ by 
 $\pm1$. In the new coordinates, denoted by $\wt x_1,\wt x_2,\wt x_3$,  the projective billiard on the 
 complement $C\setminus\{\wt x_3=0\}$ is the tautological projection of a planar billiard, as in the discussion above.
 \end{remark}
  
  \section{Acknowledgments}
  I am grateful to Misha Bialy, Andrey Mironov and Sergei Tabachnikov for helpful discussions.

 \end{document}